\documentclass[11pt,a4paper]{amsart}
\usepackage{amsmath,amssymb, amsbsy}
\usepackage{subfigure}
\usepackage{graphpap,latexsym,epsf}
\usepackage{color,psfrag}
\usepackage[dvips]{graphicx}
\usepackage{enumerate}
\usepackage{bbm}
\usepackage{relsize}
\begin{document}
\newcommand{\dyle}{\displaystyle}
\newcommand{\R}{{\mathbb{R}}}
\newcommand{\Hi}{{\mathbb H}}
\newcommand{\Ss}{{\mathbb S}}
\newcommand{\N}{{\mathbb N}}
\newcommand{\Rn}{{\mathbb{R}^n}}
\newcommand{\ieq}{\begin{equation}}
\newcommand{\eeq}{\end{equation}}
\newcommand{\ieqa}{\begin{eqnarray}}
\newcommand{\eeqa}{\end{eqnarray}}
\newcommand{\ieqas}{\begin{eqnarray*}}
\newcommand{\eeqas}{\end{eqnarray*}}
\newcommand{\Bo}{\put(260,0){\rule{2mm}{2mm}}\\}


\theoremstyle{plain}
\newtheorem{theorem}{Theorem} [section]
\newtheorem{corollary}[theorem]{Corollary}
\newtheorem{lemma}[theorem]{Lemma}
\newtheorem{proposition}[theorem]{Proposition}
\def\neweq#1{\begin{equation}\label{#1}}
\def\endeq{\end{equation}}
\def\eq#1{(\ref{#1})}


\theoremstyle{definition}
\newtheorem{definition}[theorem]{Definition}
\newtheorem{remark}[theorem]{Remark}
\numberwithin{figure}{section}

\title[Comparison between Skorokhod \& Russo-Vallois integration]{A simple comparison between Skorokhod \& \\
Russo-Vallois integration for insider trading}

\author[C. Escudero]{Carlos Escudero}
\address{}
\email{}


\date{\today}

\begin{abstract}
We consider a simplified version of the problem of insider trading in a financial market. We approach it by means of anticipating
stochastic calculus and compare the use of the Skorokhod and the Russo-Vallois forward integrals within this context.
We conclude that, while the forward integral yields results with a clear financial meaning, the Skorokhod integral
does not provide a suitable formulation for this problem.
\end{abstract}
\maketitle

\section{Introduction}

The stochastic differential equation
\begin{equation}\label{white}
\frac{dx}{dt}= f(x) + g(x) \, \xi(t),
\end{equation}
where $\xi(t)$ is a ``white noise'', is a mathematical model with applications in many disciplines~\cite{hl}. The precise meaning
of this equation is found via the introduction of a suitable stochastic integral, that can be either It\^o:
\begin{equation}\label{ito}
dx = f(x) \, dt + g(x) \, dB_t,
\end{equation}
where $B_t$ is a Brownian motion, Stratonovich:
\begin{equation}\label{str}
dx = f(x) \, dt + g(x) \circ dB_t,
\end{equation}
or yet another option~\cite{k}. The mathematical theory for stochastic differential equations of It\^{o} or Stratonovich type has been
constructed~\cite{oksendal} and both problems are shown to be well-posed under reasonable conditions, then minimizing from a pure mathematics
viewpoint the difference between them. However, from an applied viewpoint the difference between equations~\eqref{ito} and~\eqref{str} can be dramatic,
as both may lead to radically different dynamics~\cite{hl}. Which \emph{interpretation of noise} is chosen depends on modeling, that is,
on the particular application which mathematical treatment leads to equation~\eqref{white}. Perhaps because of this, a vast literature
on which is the right interpretation does exist~\cite{mmcc}.

One of the main applications of the theory of stochastic differential equations is the study of financial markets.
Let us consider a classical financial market with one asset free of risk (the bond)
\begin{eqnarray}\label{s0}
dS_0 &=& \rho \, S_0 \, dt, \\ \nonumber
S_0(0) &=& M_0,
\end{eqnarray}
and a risky asset (the stock) modeled by geometric Brownian motion
\begin{eqnarray}\label{s1}
dS_1 &=& \mu \, S_1 \, dt + \sigma \, S_1 \, dB_t, \\ \nonumber
S_1(0) &=& M_1,
\end{eqnarray}
where the constants $M_0, M_1, \rho, \mu, \sigma \in \mathbb{R}^+ := ]0,\infty[$ have the following financial meaning:
\begin{itemize}
\item $M_0$ is the initial wealth to be invested in the bond.
\item $M_1$ is the initial wealth to be invested in the stock.
\item $\rho$ is the interest rate of the bond.
\item $\mu$ is the appreciation rate of the stock.
\item $\sigma$ is the volatility of the stock.
\end{itemize}
The total initial wealth is $M = M_0+M_1$ and we assume that $\mu > \rho$. We consider the trader possesses
a fixed total initial wealth $M$ at the initial time $t=0$ and is free to choose what fraction of it, $M_0$ and $M_1$, is invested in each asset.
Clearly, at any time $t>0$, the total wealth is given by
$$
S(t)=S_0(t)+S_1(t).
$$
We will consider this financial market on $[0,T]$ for a fixed future time $T>0$.
Then we have the following result.

\begin{theorem}\label{inith}
The expected value of the total wealth at time $T$ is
\begin{equation}\nonumber
\mathbb{E}[S(T)] = M_0 \, e^{\rho T} + M_1 \, e^{\mu T}.
\end{equation}
\end{theorem}

\begin{proof}
Using It\^{o} calculus we solve equations~\eqref{s0} and~\eqref{s1} to find
\begin{eqnarray}\nonumber
S_0(t) &=& M_0 \, e^{\rho t}, \\ \nonumber
S_1(t) &=& M_1 \exp \left\{ \left( \mu -\frac{\sigma^2}{2} \right) t + \sigma B_t \right\},
\end{eqnarray}
and then the expectation of $S(t)$ at time $t=T$ is
\begin{eqnarray}\nonumber
\mathbb{E}[S(T)] &=& \mathbb{E}[S_0(T)] + \mathbb{E}[S_1(T)] \\ \nonumber
&=& M_0 \, \mathbb{E}\left[e^{\rho T}\right] +
M_1 \, \mathbb{E}\left[\exp \left\{ \left( \mu -\frac{\sigma^2}{2} \right) T + \sigma B_T \right\}\right] \\ \nonumber
&=& M_0 \, e^{\rho T} + M_1 \, e^{\mu T}.
\end{eqnarray}
\end{proof}

Any trader that wants to maximize the expected wealth at time $T$ should obviously choose the strategy
$$
M_0 = 0, \qquad M_1 = M,
$$
what in turn yields the maximal expected wealth
$$
\mathbb{E}[S(T)] = M e^{\mu T}.
$$

Some remarks are now in order.
First of all, this maximization problem may be regarded as a toy model for the Merton portfolio optimization problem~\cite{merton}.
Indeed, everything here becomes simplified due to the absence of a utility function modeling risk aversion. This function has not been introduced
for two reasons: to keep our approach and results as simple as possible, and also for some modeling reasons that will be specified in the next
section. Additionally, it is important to remark that problem~\eqref{s1} represents an easy example of the resolution of the
\emph{It\^{o} versus Stratonovich dilemma} referred to in the first paragraph of this Introduction.
In our modeling of the stock price evolution we assumed that $\mu$ is the expected rate of return of the risky asset. Therefore,
this assumption together with the martingale property of the It\^o integral, they impose unambiguously that~\eqref{s1} is
an It\^o stochastic differential equation. Things will be different in the next section, in which the trader will be assumed to posses
at time $t=0$ additional information with respect to the one contained in the filtration generated by $B_t$.

\section{Insider trading with full information}

The problem of discerning the strategies of a dishonest trader who possesses privileged information in a financial market, ``the insider'',
is a venerable one in the field of stochastic analysis applied to finance~\cite{bo,noep,jyc,leon,nualart,pk}
and continues to be of current interest~\cite{do1,do2,do3}.
Within this work, a much simplified version of this problem is considered, as our goal is to favor the accessibility to the comparison
between the two anticipating stochastic integrals in the context of finance.

Consider now that, contrary to the situation in the previous section, our trader is an insider with full information on the future price of the stock.
Precisely, the insider trader knows already at the initial time $t=0$ what will the value $S_1(T)$ be. Then
the chosen strategy should be different:
$$
M_0 = M \, \mathlarger{\mathlarger{\mathbbm{1}}} \{\bar{S}_1(T) \le \bar{S}_0(T)\},
\qquad M_1 = M \, \mathlarger{\mathlarger{\mathbbm{1}}} \{\bar{S}_1(T) > \bar{S}_0(T)\},
$$
for
\begin{eqnarray}\nonumber
d\bar{S}_0 &=& \rho \, \bar{S}_0 \, dt, \\ \nonumber
\bar{S}_0(0) &=& 1,
\end{eqnarray}
and
\begin{eqnarray}\nonumber
d\bar{S}_1 &=& \mu \, \bar{S}_1 \, dt + \sigma \, \bar{S}_1 \, dB_t, \\ \nonumber
\bar{S}_1(0) &=& 1,
\end{eqnarray}
that is, the insider always bets the most profitable asset. It is then natural to ask what would
be the expected wealth of the insider at time $T$. Note again that we are not considering any utility function
modeling risk aversion. This is, as mentioned in the Introduction, in part for the sake of simplicity and in part
for modeling reasons: it is not clear what the role of risk aversion should be in the case of an insider with full information
on the future value of the stock. In order to answer this question we note that, while
the initial value problem
\begin{subequations}
\begin{eqnarray}\label{rode1}
dS_0 &=& \rho \, S_0 \, dt, \\ \label{rode2}
S_0(0) &=& M \, \mathlarger{\mathlarger{\mathbbm{1}}} \{\bar{S}_1(T) \le \bar{S}_0(T)\},
\end{eqnarray}
\end{subequations}
is an ordinary differential equation with a random initial condition, the problem
\begin{eqnarray}\nonumber
dS_1 &=& \mu \, S_1 \, dt + \sigma \, S_1 \, dB_t, \\ \nonumber
S_1(0) &=& M \, \mathlarger{\mathlarger{\mathbbm{1}}} \{\bar{S}_1(T) > \bar{S}_0(T)\},
\end{eqnarray}
is ill-posed as an It\^o stochastic differential equation. This is because the initial condition is anticipating, and this
anticipating character will propagate into the solution, therefore giving rise to the It\^o integral of a non-adapted integrand, which is
of course meaningless. One way to circumvent this pitfall is replacing the It\^o integral in our model by one of its generalizations that
admit non-adapted integrands. Two possibilities are the Skorokhod integral~\cite{skorokhod} and the Russo-Vallois forward integral~\cite{russovallois}.
Both integrals reduce to the It\^o one when the integrand is adapted, but are different in general~\cite{noep}.

Using established notation~\cite{noep}, and choosing the Skorokhod integral,
we arrive at the initial value problem
\begin{subequations}
\begin{eqnarray} \label{sko1}
\delta S_1 &=& \mu \, S_1 \, d t + \sigma \, S_1 \, \delta B_t \\ \label{sko2}
S_1(0) &=& M \, \mathlarger{\mathlarger{\mathbbm{1}}} \{\bar{S}_1(T) > \bar{S}_0(T)\},
\end{eqnarray}
\end{subequations}
for a Skorokhod stochastic differential equation.
Analogously, when the choice is the Russo-Vallois integral, we face the initial value problem
\begin{subequations}
\begin{eqnarray} \label{rv1}
d^- S_1 &=& \mu \, S_1 \, dt + \sigma \, S_1 \, d^- B_t \\ \label{rv2}
S_1(0) &=& M \, \mathlarger{\mathlarger{\mathbbm{1}}} \{\bar{S}_1(T) > \bar{S}_0(T)\},
\end{eqnarray}
\end{subequations}
for a forward stochastic differential equation. As happened with the It\^o versus Stratonovich dilemma described in the Introduction, it is
in principle possible to choose either equation~\eqref{sko1} or~\eqref{rv1} to address the problem at hand. As in this classical situation,
both equations~\eqref{sko1} and~\eqref{rv1} are well-founded theoretically~\cite{lssdes,noep,leon,nualart},
so only the particular applications will dictate which is
the ``right interpretation of noise''. Since we are addressing a financial problem, we will unveil the right choice in this concrete case.
To this end we need the following result that describes the time behavior of systems~\eqref{sko1}-\eqref{sko2} and~\eqref{rv1}-\eqref{rv2}.
We remind the reader that the total wealth of the insider is still given by
$$
S(t)=S_0(t)+S_1(t).
$$

\begin{theorem}\label{mainth}
The expected value of the total wealth of the insider at time $t=T$ is
\begin{eqnarray}\nonumber
\mathbb{E}[S(T)] &=& \frac{M}{2} \left\{ 1 + \text{erf} \left[ \frac{(\sigma^2 + 2 \rho - 2 \mu)\sqrt{T}}{2 \sqrt{2} \, \sigma} \right] \right\} e^{\rho T}
\\ \nonumber & & + \,
\frac{M}{2} \left\{ 1 - \text{erf} \left[ \frac{(\sigma^2 + 2 \rho - 2 \mu)\sqrt{T}}{2 \sqrt{2} \, \sigma} \right] \right\} e^{\mu T},
\end{eqnarray}
for model~\eqref{rode1}-\eqref{rode2} and~\eqref{sko1}-\eqref{sko2}, while it is
\begin{eqnarray}\nonumber
\mathbb{E}[S(T)] &=& \frac{M}{2} \left\{ 1 + \text{erf} \left[ \frac{(\sigma^2 + 2 \rho - 2 \mu)\sqrt{T}}{2 \sqrt{2} \, \sigma} \right] \right\} e^{\rho T}
\\ \nonumber & & + \,
\frac{M}{2} \left\{ 1 + \text{erf} \left[ \frac{(\sigma^2 - 2 \rho + 2 \mu)\sqrt{T}}{2 \sqrt{2} \, \sigma} \right] \right\} e^{\mu T},
\end{eqnarray}
for model~\eqref{rode1}-\eqref{rode2} and~\eqref{rv1}-\eqref{rv2},
where
$$
\text{erf}\,(\cdot)= \frac{2}{\sqrt{\pi}} \int_0^\cdot \! e^{-x^2} \, dx
$$
is the error function.
\end{theorem}

\begin{proof}
Using Malliavin calculus techniques~\cite{noep} it is possible to solve problem~\eqref{sko1}-\eqref{sko2} explicitly to find
$$
S_1(t)= M \, \mathlarger{\mathlarger{\mathbbm{1}}} \{\bar{S}_1(T) > \bar{S}_0(T)\} \diamond
\exp \left\{ \left( \mu -\frac{\sigma^2}{2} \right) t + \sigma B_t \right\},
$$
where $\diamond$ denotes the Wick product~\cite{noep}.
Now, using the factorization property of the expectation of a Wick product of random variables,
we find for the expected wealth at the terminal time:
\begin{eqnarray}\nonumber
\mathbb{E}[S(T)] &=& \mathbb{E}[S_0(T)] + \mathbb{E}[S_1(T)] \\ \nonumber
&=& M \, \mathbb{E}\left[ \mathlarger{\mathlarger{\mathbbm{1}}} \{\bar{S}_1(T) \le \bar{S}_0(T)\} \right] e^{\rho T} \\ \nonumber
& & + \, M \, \mathbb{E} \left[ \mathlarger{\mathlarger{\mathbbm{1}}} \{\bar{S}_1(T) > \bar{S}_0(T)\} \diamond
\exp \left\{ \left( \mu -\frac{\sigma^2}{2} \right) T + \sigma B_T \right\} \right] \\ \nonumber
&=& M \, \mathbb{E} \left[ \mathlarger{\mathlarger{\mathbbm{1}}} \{\bar{S}_1(T) \le \bar{S}_0(T)\}\right] e^{\rho T} \\ \nonumber
& & + \, M \, \mathbb{E} \left[ \mathlarger{\mathlarger{\mathbbm{1}}} \{\bar{S}_1(T) > \bar{S}_0(T)\} \right]
\mathbb{E} \left[\exp \left\{ \left( \mu -\frac{\sigma^2}{2} \right) T + \sigma B_T \right\} \right] \\ \nonumber
&=& M \, \text{Pr}\left\{\bar{S}_1(T) \le \bar{S}_0(T)\right\} e^{\rho T}
+ M \, \text{Pr}\left\{\bar{S}_1(T) > \bar{S}_0(T)\right\} e^{\mu T} \\ \nonumber
&=& M \, \text{Pr}\left\{ B_T \le (\rho -\mu +\sigma^2/2)T/\sigma \right\} e^{\rho T}
+ M \, \text{Pr}\left\{ B_T > (\rho -\mu +\sigma^2/2)T/\sigma \right\} e^{\mu T} \\ \nonumber
&=& \frac{M}{2} \left\{ 1 + \text{erf} \left[ \frac{(\sigma^2 + 2 \rho - 2 \mu)\sqrt{T}}{2 \sqrt{2} \, \sigma} \right] \right\} e^{\rho T}
\\ \nonumber & & + \,
\frac{M}{2} \left\{ 1 - \text{erf} \left[ \frac{(\sigma^2 + 2 \rho - 2 \mu)\sqrt{T}}{2 \sqrt{2} \, \sigma} \right] \right\} e^{\mu T},
\end{eqnarray}
where we have also used that $B_T \sim \mathcal{N}(0,T)$.

Since the forward integral preserves It\^o calculus~\cite{noep}, the solution to problem~\eqref{rv1}-\eqref{rv2} can be computed
using It\^o calculus rules:
$$
S_1(t)= M \, \mathlarger{\mathlarger{\mathbbm{1}}} \{\bar{S}_1(T) > \bar{S}_0(T)\}
\exp \left\{ \left( \mu -\frac{\sigma^2}{2} \right) t + \sigma B_t \right\}.
$$
Therefore the expected wealth at the terminal time in this case is
\begin{eqnarray}\nonumber
\mathbb{E}[S(T)] &=& \mathbb{E}[S_0(T)] + \mathbb{E}[S_1(T)] \\ \nonumber
&=& M \, \mathbb{E}\left[ \mathlarger{\mathlarger{\mathbbm{1}}} \{\bar{S}_1(T) \le \bar{S}_0(T)\} \right] e^{\rho T} \\ \nonumber
& & + \, M \, \mathbb{E} \left[ \mathlarger{\mathlarger{\mathbbm{1}}} \{\bar{S}_1(T) > \bar{S}_0(T)\}
\exp \left\{ \left( \mu -\frac{\sigma^2}{2} \right) T + \sigma B_T \right\} \right] \\ \nonumber
&=& M \, \text{Pr} \{\bar{S}_1(T) \le \bar{S}_0(T)\} \, e^{\rho T} \\ \nonumber
& & + \, M \, \exp \left\{ \left( \mu -\frac{\sigma^2}{2} \right) T \right\}
\mathbb{E} \left[ \mathlarger{\mathlarger{\mathbbm{1}}} \{\bar{S}_1(T) > \bar{S}_0(T)\}
\exp \left\{ \sigma B_T \right\} \right] \\ \nonumber
&=& M \, \text{Pr} \{ B_T \le (\rho -\mu +\sigma^2/2)T/\sigma \} \, e^{\rho T} \\ \nonumber
& & + \, M \, \exp \left\{ \left( \mu -\frac{\sigma^2}{2} \right) T \right\}
\mathbb{E} \left[ \mathlarger{\mathlarger{\mathbbm{1}}} \{ B_T > (\rho -\mu +\sigma^2/2)T/\sigma \}
\exp \left\{ \sigma B_T \right\} \right] \\ \nonumber
&=& \frac{M}{2} \left\{ 1 + \text{erf} \left[ \frac{(\sigma^2 + 2 \rho - 2 \mu)\sqrt{T}}{2 \sqrt{2} \, \sigma} \right] \right\} e^{\rho T}
\\ \nonumber & & + \,
\frac{M}{2} \left\{ 1 + \text{erf} \left[ \frac{(\sigma^2 - 2 \rho + 2 \mu)\sqrt{T}}{2 \sqrt{2} \, \sigma} \right] \right\} e^{\mu T}.
\end{eqnarray}
\end{proof}

\section{Consequences}

The problem of insider trading has been approached by means of the use of the forward integral~\cite{bo,noep,do1,do2,do3,leon,nualart},
but the justification of this choice has been usually made on more technical grounds.
A financial justification of the use of the forward integral was however illustrated in~\cite{bo} with a buy-and-hold strategy.
The following result further supports this choice and it is purely based on a direct comparison between the financial consequences of
employing either integral.

\begin{theorem}
Let us denote by $S^{(\text{i})}(t)$ the total wealth process corresponding to the initial value problems~\eqref{s0} subject to $M_0=0$
and~\eqref{s1} subject to $M_1=M$; denote also
by $S^{(\text{sk})}(t)$ and $S^{(\text{rs})}(t)$ the total wealth processes corresponding to the initial value problems~\eqref{rode1}-\eqref{rode2} and~\eqref{sko1}-\eqref{sko2}, and ~\eqref{rode1}-\eqref{rode2} and~\eqref{rv1}-\eqref{rv2}, respectively. Then
$$
\mathbb{E}[S^{(\text{sk})}(T)] < \mathbb{E}[S^{(\text{i})}(T)] < \mathbb{E}[S^{(\text{rs})}(T)],
$$
for any $M, \rho, \mu, \sigma, T \in \mathbb{R}^+$ with $\mu > \rho$.
\end{theorem}

\begin{proof}
From our previous results it it clear that
$$
\mathbb{E}[S^{(\text{i})}(T)] = M e^{\mu T}
$$
and
\begin{eqnarray}\nonumber
\mathbb{E}[S^{(\text{sk})}(T)] &=& \frac{M}{2} \left\{ 1 + \text{erf} \left[ \frac{(\sigma^2 + 2 \rho - 2 \mu)\sqrt{T}}{2 \sqrt{2} \, \sigma} \right] \right\}
e^{\rho T} \\ \nonumber & & + \,
\frac{M}{2} \left\{ 1 - \text{erf} \left[ \frac{(\sigma^2 + 2 \rho - 2 \mu)\sqrt{T}}{2 \sqrt{2} \, \sigma} \right] \right\} e^{\mu T}.
\end{eqnarray}
The inequality
\begin{eqnarray}\nonumber
M e^{\mu T}&>& \frac{M}{2} \left\{ 1 + \text{erf} \left[ \frac{(\sigma^2 + 2 \rho - 2 \mu)\sqrt{T}}{2 \sqrt{2} \, \sigma} \right] \right\} e^{\rho T}
\\ \nonumber & & + \,
\frac{M}{2} \left\{ 1 - \text{erf} \left[ \frac{(\sigma^2 + 2 \rho - 2 \mu)\sqrt{T}}{2 \sqrt{2} \, \sigma} \right] \right\} e^{\mu T},
\end{eqnarray}
whenever $\mu > \rho$, follows directly from the definition of the error function~\cite{as}.

On the other hand, from the proof of Theorem~\ref{mainth} we find that
\begin{eqnarray}\nonumber
\mathbb{E}[S^{(\text{rs})}(T)]
&=& M \, \mathbb{E}\left[ \mathlarger{\mathlarger{\mathbbm{1}}} \{B_T \le (\rho -\mu +\sigma^2/2)T/\sigma\} \, e^{\rho T} \right] \\ \nonumber
& & + \, M \, \mathbb{E} \left[ \mathlarger{\mathlarger{\mathbbm{1}}} \{B_T > (\rho -\mu +\sigma^2/2)T/\sigma\}
\exp \left\{ \left( \mu -\frac{\sigma^2}{2} \right) T + \sigma B_T \right\} \right] \\ \nonumber
&=& M \, \mathbb{E}\left[ \mathlarger{\mathlarger{\mathbbm{1}}} \{B_T < (\rho -\mu +\sigma^2/2)T/\sigma\} \, e^{\rho T} \right] \\ \nonumber
& & + \, M \, \mathbb{E} \left[ \mathlarger{\mathlarger{\mathbbm{1}}} \{B_T > (\rho -\mu +\sigma^2/2)T/\sigma\}
\exp \left\{ \left( \mu -\frac{\sigma^2}{2} \right) T + \sigma B_T \right\} \right] \\ \nonumber
&>& M \, \mathbb{E}\left[ \mathlarger{\mathlarger{\mathbbm{1}}} \{B_T < (\rho -\mu +\sigma^2/2)T/\sigma\}
\exp \left\{ \left( \mu -\frac{\sigma^2}{2} \right) T + \sigma B_T \right\} \right] \\ \nonumber
& & + \, M \, \mathbb{E} \left[ \mathlarger{\mathlarger{\mathbbm{1}}} \{B_T > (\rho -\mu +\sigma^2/2)T/\sigma\}
\exp \left\{ \left( \mu -\frac{\sigma^2}{2} \right) T + \sigma B_T \right\} \right] \\ \nonumber
&=& M \, \mathbb{E}\left[ \mathlarger{\mathlarger{\mathbbm{1}}} \{B_T \le (\rho -\mu +\sigma^2/2)T/\sigma\}
\exp \left\{ \left( \mu -\frac{\sigma^2}{2} \right) T + \sigma B_T \right\} \right] \\ \nonumber
& & + \, M \, \mathbb{E} \left[ \mathlarger{\mathlarger{\mathbbm{1}}} \{B_T > (\rho -\mu +\sigma^2/2)T/\sigma\}
\exp \left\{ \left( \mu -\frac{\sigma^2}{2} \right) T + \sigma B_T \right\} \right] \\ \nonumber
&=& M \, \mathbb{E} \left[ \exp \left\{ \left( \mu -\frac{\sigma^2}{2} \right) T + \sigma B_T \right\} \right] \\ \nonumber
&=& M \, e^{\mu T}.
\end{eqnarray}
\end{proof}

Our results illustrate that the forward integral provides results with a clear financial meaning, at least in this context. On the other hand
the Skorokhod integral yields a result that is meaningless from the financial viewpoint, as the expected wealth of the
insider at the terminal time under this model is less than the corresponding wealth of the honest trader.

\section{Further results}

So far we have used the hypothesis $\mu > \rho$. This is a modeling assumption: the expected return of a risky investment should be higher than
that of a riskless investment in order to attract investors. From a mathematical viewpoint one can consider the reciprocal case $\rho > \mu$
and still obtain a result in the same line to that in the previous section. Note that in this new scenario the honest trader will obviously
choose the strategy $M_0=M$ and $M_1=0$, and the corresponding wealth will be $S(T)=M e^{\rho T}$.

\begin{theorem}
Let us denote by $S^{(\text{i})}(t)$ the total wealth process corresponding to the initial value problems~\eqref{s0} subject to $M_0=M$
and~\eqref{s1} subject to $M_1=0$; also denote
by $S^{(\text{sk})}(t)$ and $S^{(\text{rs})}(t)$ the total wealth processes corresponding to the initial value problems~\eqref{rode1}-\eqref{rode2} and~\eqref{sko1}-\eqref{sko2}, and ~\eqref{rode1}-\eqref{rode2} and~\eqref{rv1}-\eqref{rv2}, respectively. Then
$$
\mathbb{E}[S^{(\text{sk})}(T)] < \mathbb{E}[S^{(\text{i})}(T)] < \mathbb{E}[S^{(\text{rs})}(T)],
$$
for any $M, \rho, \mu, \sigma, T \in \mathbb{R}^+$ with $\rho > \mu$.
\end{theorem}

\begin{proof}
From Theorems~\ref{inith} and~\ref{mainth} it follows that
$$
\mathbb{E}[S^{(\text{i})}(T)] = S^{(\text{i})}(T) = M e^{\rho T}
$$
and
\begin{eqnarray}\nonumber
\mathbb{E}[S^{(\text{sk})}(T)] &=& \frac{M}{2} \left\{ 1 + \text{erf} \left[ \frac{(\sigma^2 + 2 \rho - 2 \mu)\sqrt{T}}{2 \sqrt{2} \, \sigma} \right] \right\}
e^{\rho T} \\ \nonumber & & + \,
\frac{M}{2} \left\{ 1 - \text{erf} \left[ \frac{(\sigma^2 + 2 \rho - 2 \mu)\sqrt{T}}{2 \sqrt{2} \, \sigma} \right] \right\} e^{\mu T}.
\end{eqnarray}
The inequality
\begin{eqnarray}\nonumber
M e^{\rho T}&>& \frac{M}{2} \left\{ 1 + \text{erf} \left[ \frac{(\sigma^2 + 2 \rho - 2 \mu)\sqrt{T}}{2 \sqrt{2} \, \sigma} \right] \right\} e^{\rho T}
\\ \nonumber & & + \,
\frac{M}{2} \left\{ 1 - \text{erf} \left[ \frac{(\sigma^2 + 2 \rho - 2 \mu)\sqrt{T}}{2 \sqrt{2} \, \sigma} \right] \right\} e^{\mu T},
\end{eqnarray}
for $\rho > \mu$, is a direct consequence of the definition of the error function~\cite{as}.

Now, from the proof of Theorem~\ref{mainth} we see that
\begin{eqnarray}\nonumber
\mathbb{E}[S^{(\text{rs})}(T)]
&=& M \, \mathbb{E}\left[ \mathlarger{\mathlarger{\mathbbm{1}}} \{B_T \le (\rho -\mu +\sigma^2/2)T/\sigma\} \, e^{\rho T} \right] \\ \nonumber
& & + \, M \, \mathbb{E} \left[ \mathlarger{\mathlarger{\mathbbm{1}}} \{B_T > (\rho -\mu +\sigma^2/2)T/\sigma\}
\exp \left\{ \left( \mu -\frac{\sigma^2}{2} \right) T + \sigma B_T \right\} \right] \\ \nonumber
&>& M \, \mathbb{E}\left[ \mathlarger{\mathlarger{\mathbbm{1}}} \{B_T \le (\rho -\mu +\sigma^2/2)T/\sigma\}
e^{\rho T} \right] \\ \nonumber
& & + \, M \, \mathbb{E} \left[ \mathlarger{\mathlarger{\mathbbm{1}}} \{B_T > (\rho -\mu +\sigma^2/2)T/\sigma\}
e^{\rho T} \right] \\ \nonumber
&=& M \, \mathbb{E} \left[ e^{\rho T} \right] \\ \nonumber
&=& M \, e^{\rho T}.
\end{eqnarray}
\end{proof}

The marginal case $\mu = \rho$ can be analyzed along the same way. Then the expected wealth of the honest trader will be the same
independently of the initial strategy employed: $\mathbb{E}[S^{(\text{i})}(T)] = M e^{\rho T}$. In this case the following result follows.

\begin{theorem}
For $\mu= \rho$ the expected values of the total wealth processes fulfil
\begin{eqnarray}\nonumber
\mathbb{E}[S^{(\text{i})}(T)] &=& M e^{\rho T}, \\ \nonumber
\mathbb{E}[S^{(\text{sk})}(T)] &=& M e^{\rho T}, \\ \nonumber
\mathbb{E}[S^{(\text{rs})}(T)] &=& M \left[ 1 + \text{erf} \left( \frac{\sigma \, \sqrt{T}}{2 \sqrt{2}} \right) \right] e^{\rho T},
\end{eqnarray}
and therefore
$$
\mathbb{E}[S^{(\text{sk})}(T)] = \mathbb{E}[S^{(\text{i})}(T)] < \mathbb{E}[S^{(\text{rs})}(T)].
$$
\end{theorem}

\begin{proof}
By Theorem~\ref{inith}, for the honest trader we compute
\begin{eqnarray}\nonumber
\mathbb{E}[S^{(\text{i})}(T)] &=& M_0 e^{\rho T} + M_1 e^{\mu T}, \\ \nonumber
&=& \left( M_0 + M_1 \right) e^{\rho T}, \\ \nonumber
&=& M e^{\rho T}.
\end{eqnarray}
From Theorem~\ref{mainth} we find for the Skorokhod insider:
\begin{eqnarray}\nonumber
\mathbb{E}[S^{(\text{sk})}(T)] &=& \frac{M}{2} \left\{ 1 + \text{erf} \left[ \frac{(\sigma^2 + 2 \rho - 2 \mu)\sqrt{T}}{2 \sqrt{2} \, \sigma} \right] \right\} e^{\rho T}
\\ \nonumber & & + \,
\frac{M}{2} \left\{ 1 - \text{erf} \left[ \frac{(\sigma^2 + 2 \rho - 2 \mu)\sqrt{T}}{2 \sqrt{2} \, \sigma} \right] \right\} e^{\mu T} \\ \nonumber
&=& \frac{M}{2} \left[ 1 + \text{erf} \left( \frac{\sigma \, \sqrt{T}}{2 \sqrt{2}} \right) \right] e^{\rho T} + \,
\frac{M}{2} \left[ 1 - \text{erf} \left( \frac{\sigma \, \sqrt{T}}{2 \sqrt{2} \, \sigma} \right) \right] e^{\rho T} \\ \nonumber
&=& M e^{\rho T}.
\end{eqnarray}
Again from Theorem~\ref{mainth} we can compute the expected wealth of the Russo-Vallois insider; it is
\begin{eqnarray}\nonumber
\mathbb{E}[S^{(\text{rs})}(T)] &=& \frac{M}{2} \left\{ 1 + \text{erf} \left[ \frac{(\sigma^2 + 2 \rho - 2 \mu)\sqrt{T}}{2 \sqrt{2} \, \sigma} \right] \right\} e^{\rho T}
\\ \nonumber & & + \,
\frac{M}{2} \left\{ 1 + \text{erf} \left[ \frac{(\sigma^2 - 2 \rho + 2 \mu)\sqrt{T}}{2 \sqrt{2} \, \sigma} \right] \right\} e^{\mu T} \\ \nonumber
&=& M \left[ 1 + \text{erf} \left( \frac{\sigma \, \sqrt{T}}{2 \sqrt{2}} \right) \right] e^{\rho T}.
\end{eqnarray}
Then the equality $\mathbb{E}[S^{(\text{sk})}(T)] = \mathbb{E}[S^{(\text{i})}(T)]$ is immediate
and the inequality $\mathbb{E}[S^{(\text{i})}(T)] < \mathbb{E}[S^{(\text{rs})}(T)]$ is a simple consequence of the definition
of the error function~\cite{as}.
\end{proof}

The results in this section, although they are perhaps of a weaker financial meaning, again reveal the same fact:
the Skorokhod integral, when used to model insider trading, presents paradoxes that are not present in the models interpreted
according to the Russo-Vallois forward integral.

\section{Outlook}

Making precise a stochastic differential equation model by means of choosing a suitable stochastic integral is a topic that has
received much attention in the physical literature~\cite{mmcc}. This choice does not usually change the well-posedness of the problem,
but may modify abruptly the dynamics of the equation. Therefore the selection should be based on modeling assumptions, and of course
any particular choice is strongly model-dependent. While historically the discussion has focused on the non-anticipating framework and
the It\^o/Stratonovich duality, there is nothing substantially different between this case and the anticipating one in this respect.
Therefore the question of interpreting a given anticipating stochastic differential equation in the Skorokhod or Russo-Vallois sense
falls in this category. Our present results point to the fact that the Russo-Vallois forward integral is well-adapted for modeling insider
trading in a financial market, but the Skorokhod integral is not suitable for this purpose. This of course does not affect the fact that
both types of anticipating stochastic differential equation are well-defined, and that presumably the ``Skorokhod interpration of noise''
will be of use in other applications, be them financial, physical, or yet others. Time will reveal which anticipating stochastic integrals are
useful in different applications, just like the applications the It\^o and Stratonovich stochastic integrals
are useful for have been revealed along the years.

\section*{Acknowledgements}

The author is grateful to Bernt {\O}ksendal and Olfa Draouil for helpful discussions and comments.
This work has been partially supported by the Government of Spain (Ministry of Economy, Industry and Competitiveness) through Project MTM2015-72907-EXP.

\vskip5mm
\noindent
{\footnotesize
Carlos Escudero\par\noindent
Departamento de Matem\'aticas\par\noindent
Universidad Aut\'onoma de Madrid\par\noindent
{\tt carlos.escudero@uam.es}\par\vskip1mm\noindent
}
\end{document}